\documentclass[a4paper,10pt]{amsart}

\usepackage[latin1]{inputenc}
\usepackage[T1]{fontenc}
\usepackage{mathrsfs,amssymb,amsthm,amsmath,verbatim}
\usepackage[usenames]{color}
\usepackage{t1enc}
\usepackage{enumerate}

\usepackage[left=3.5 cm, right=3.5cm]{geometry}
\newcommand{\ulam}{\textbf{U}}
\newcommand{\mrow}{\textbf{M}}
\newcommand{\esep}{\textbf{K}}
\newcommand{\smf}{\hspace{0.008 cm}^\smallfrown}
\newcommand{\dom}{\text{dom}}
\newcommand{\ran}{\text{ran}}
\newcommand{\card}{\text{Card}}
\newcommand{\mb}[1]{\mathbb{#1}}
\newcommand{\mc}[1]{\mathcal{#1}}
\newcommand{\oo}{\omega}
\newcommand{\setm}{\setminus}

\newcommand{\omg}{{\omega_1}}
\newcommand{\uhp}{\upharpoonright}

\title{On spaces with $\sigma$-closed-discrete dense sets}
\author{Rodrigo R. Dias, Daniel T. Soukup}
\date{\today}

\address{Universidade Federal do ABC,
  Avenida dos Estados, 5001,
  Santo Andr\'e, SP,
  09210-580,
  Brazil}
\email{rodrigo.dias@ufabc.edu.br}

\address{Universit\"at Wien
Kurt G\"odel Research Center for Mathematical Logic
W\"ahringer Strasse 25
1090 WIEN
AUSTRIA}
 \email[Corresponding author]{daniel.soukup@univie.ac.at}
  \urladdr{http://www.logic.univie.ac.at/~soukupd73/}

\begin{document}
\begin{abstract}
The main purpose of this paper is to study \emph{$e$-separable spaces}, originally introduced by Kurepa as $K_0'$ spaces; we call a space $X$ $e$-separable iff $X$ has a dense set which is the union of countably many closed discrete sets. We primarily focus on the behaviour of $e$-separable spaces under products and the cardinal invariants that are naturally related to $e$-separable spaces. Our main results show that the statement ``there is a product of at most $\mathfrak c$ many $e$-separable spaces that fails to be $e$-separable'' is equiconsistent with the existence of a weakly compact cardinal.
\end{abstract}
\maketitle

\newtheorem{defin}{Definition}[section]

\newtheorem{prop}[defin]{Proposition}
\newtheorem{question}[defin]{Question}
\newtheorem{obs}[defin]{Observation}

\newtheorem{prob}[defin]{Problem}

\newtheorem{lemma}[defin]{Lemma}

\newtheorem{corol}[defin]{Corollary}

\newtheorem{fact}[defin]{Fact}

\newtheorem{example}[defin]{Example}

\newtheorem{thm}[defin]{Theorem}

\newtheorem{clm}[defin]{Claim}

\newtheorem{subclaim}[defin]{Subclaim}

\newtheorem{rem}[defin]{Remark}

\makeatletter
\newtheorem*{rep@theorem}{\rep@title}
\newcommand{\newreptheorem}[2]{%
\newenvironment{rep#1}[1]{%
 \def\rep@title{#2 \ref{##1}}%
 \begin{rep@theorem}}%
 {\end{rep@theorem}}}
\makeatother

\newreptheorem{corollary}{Corollary}
\newreptheorem{theorem}{Theorem}

\section{Introduction}

The goal of this paper is to study a natural generalization of separability: let us call a space $X$ \emph{$e$-separable} iff $X$ has a dense set which is the union of countably many closed discrete sets. The definition is due to Kurepa \cite{kurepa}, who introduced this notion as property $K_0'$ in his study of Suslin's problem. Later, $e$-separable spaces appear in multiple papers related to the study of linearly ordered and GO-spaces \cite{faber, qiao,qiao2, watson}. In particular, Faber \cite{faber} showed that $e$-separable GO-spaces are perfect; however, the converse is famously open: is there, in ZFC, a perfect GO-space (or even just a perfect $T_3$ space) which is not $e$-separable? Let us refer the interested reader to a paper of Benett and Lutzer \cite{lutzer} for more details and results on this topic.

Now, our interest lies mainly in studying $e$-separability with regards to powers and products. Recall that the famous Hewitt--Marczewski--Pondiczery theorem \cite{engelking} states that the product of at most $\mathfrak c$ many separable spaces is again separable. What can we say about $e$-separable spaces in this matter? Historically, another generalization of separable spaces received more attention: \emph{$d$-separable spaces} i.e. spaces with $\sigma$-discrete dense sets. In Kurepa's old notation, $d$-separable spaces  were called $K_0$. $d$-separable spaces were investigated in great detail (see \cite{arh81, dD, juhi, moore2, weaksep, tkac}) and they show very interesting behaviour in many aspects, in particular, regarding products. A. Arhangel'ski\u\i\mbox{} proved in \cite{arh81} that any product of $d$-separable spaces is $d$-separable; in \cite{juhi}, the authors show that for every space $X$ there is a cardinal $\kappa$ so that $X^\kappa$ is $d$-separable.  Motivated by these results, one of our main objectives is to understand, as much as possible, the behaviour of $e$-separable spaces under products.

Our paper splits into three main parts. First, we make initial observations on $e$-separable spaces in Section \ref{prelim}. Then, in Section \ref{d_and_e}, we investigate if the existence of many large closed discrete sets suffices for a space to be $e$-separable. In particular, we prove that once an infinite power $X^\kappa$ has a closed discrete set of size $d(X^\kappa)$ (the density of $X^\kappa$) then $X^\kappa$ is $e$-separable. As a corollary, we show that certain large powers of non-countably-compact spaces are $e$-separable. Now an interesting open question is whether a countably compact, non-separable space can have an $e$-separable square.

Next, in Section \ref{sizes}, we compare two natural cardinal functions: $d(X)$, the size of the smallest dense set in $X$, with $d_e(X)$, the size of the smallest $\sigma$-closed-discrete dense set. In Theorem \ref{examp}, we show that there is a 0-dimensional space $X$ which satisfies $d(X)<d_e(X)$. We show that a similar example can be constructed for $d$-separable spaces, at least under $\aleph_1<\mathfrak c=2^{\aleph_0}$; we do not know how to remove this assumption. The section ends with a few interesting open problems.

Our main results are finally presented in Section \ref{pressec}: we describe those cardinals $\kappa$ such that the product of $\kappa$ many $e$-separable spaces is $e$-separable again, and hence present the analogue of the  Hewitt--Marczewski--Pondiczery theorem for $e$-separable spaces. First, note that $2^{\mathfrak c^+}$ is not $e$-separable (as a compact, non-separable space) and so the question of preserving $e$-separability comes down to products of at most $\mathfrak c$ terms. How could it be possible that $e$-separability is not preserved by small products? The reason must be that there are some large cardinals lurking in the background:


\begin{repcorollary}{main1cor} If the existence of a weakly compact cardinal is consistent with ZFC then so is the statement that there are less than $\mathfrak{c}$ many discrete spaces with non-$e$-separable product.
\end{repcorollary}


\begin{repcorollary}{main2cor}
If there is a non-$e$-separable product of at most $\mathfrak c$ many $e$-separable spaces then there is a weakly compact cardinal in $L$.
\end{repcorollary}

As we shall see, the proof of these results nicely combines various ideas from topology, set theory and logic.

\medskip

Throughout this paper, all spaces are assumed to be $T_1$. Given a product of discrete spaces $X=\prod \{X_\alpha:\alpha<\lambda\}$
and a function $\varepsilon$ satisfying $\dom(\varepsilon)\in[\lambda]^{<\aleph_0}$ and $\varepsilon(\alpha)\in X_\alpha$ for each $\alpha\in\dom(\varepsilon)$, we write $$[\varepsilon]=\{x\in X: \varepsilon\subseteq x\}.$$ Thus, if $x\in X$ is such that $x\uhp \dom(\varepsilon)=\varepsilon$, then $[\varepsilon]$ is a basic open neighbourhood of $x$ in $X$. We let $D(\kappa)$ denote the discrete space on a cardinal $\kappa$.

In general, we use standard notation and terminology consistent with Engelking \cite{engelking}.


\section{Preliminaries}\label{prelim}

The main concept we study in this paper is the following:

\begin{defin}
\label{e-sep}

A topological space $X$ is \emph{$e$-separable} if there is a
sequence $(D_n)_{n\in\omega}$ of closed discrete subspaces of $X$ such
that $\bigcup_{n\in\omega}D_n$ is dense in $X$.

\end{defin}

In this section, we will prove a few general facts about $e$-separable spaces and state some results for later reference. Let us start with simple observations:

\begin{obs} Every separable space is $e$-separable and every $e$-separable space is $d$-separable. 
\end{obs}

Recall the following two well known cardinal functions: the density of $X$, denoted by $d(X)$, is the smallest possible size of a dense set in $X$. The extent of a space $X$, denoted by $e(X)$, is the supremum of all cardinalities $|E|$ where $E$ is a closed discrete subset of $X$.

\begin{obs}\label{cardobs} Every $e$-separable space $X$ satisfies $d(X)\leq e(X)$; moreover, if $cf(d(X))>\omega$ then there is a closed discrete set of size $d(X)$ in $X$.  In particular, a countably compact space is $e$-separable iff it is separable.
\end{obs}

\begin{example}\label{2^c+} $2^{\mathfrak{c}^+}$ is a compact, $d$-separable but non-$e$-separable space.
\end{example}
\begin{proof} By Arhangel'ski\u\i{}'s \cite{arh81}, $d$-separability is preserved by products. Also, $2^{\mathfrak{c}^+}$ is not $e$-separable as $d(2^{\mathfrak{c}^+})>e(2^{\mathfrak{c}^+})=\oo$, hence Observation \ref{cardobs} can be applied.
\end{proof}

What can we say about metric spaces?

\begin{obs} Every space with a $\sigma$-discrete $\pi$-base is $e$-separable. Hence, every metrizable space is $e$-separable.
\end{obs}

The following result shows that actually a large class of generalized metric spaces are $e$-separable:

\begin{prop}
\label{developD}

Every developable space is $e$-separable.

\end{prop}

Recall that a space $X$ is \emph{developable} iff there is a developement of $X$, i.e. a sequence $(\mathscr{G}_n)_{n\in\omega}$ of open covers of $X$ such that for every $x\in X$ and open $V$ containing $x$ there is an $n\in \omega$ so that $st(x,\mathscr{G}_n)=\bigcup\{U\in \mathscr G_n:x\in U\}\subseteq V$.

\begin{proof}

Let $(\mathscr{G}_n)_{n\in\omega}$ be a development for a topological
space $X$. For each $x\in X$ and $n\in\omega$, let
$V^x_n=st(x,\mathscr{G}_n)$. By Proposition 1.3 of \cite{leandromD},
there is  a closed discrete $D_n\subseteq X$ such
that $X=\bigcup_{x\in D_n}V^x_n$ for each $n\in\omega$. We claim that $\bigcup_{n\in\omega}D_n$
is dense in $X$.

Suppose, to the contrary, that there is $p\in
X\setminus\overline{\bigcup_{n\in\omega}D_n}$, and let $m\in\omega$ be
such that $V^p_m\cap\bigcup_{n\in\omega}D_n=\emptyset$. By the choice
of $D_m$, there is $x\in D_m$ such that $p\in V^x_m$; but then
$\{p,x\}\subseteq U$ for some $U\in\mathscr{G}_m$, which implies that
$x\in V^p_m$, thus contradicting the fact that $V^p_m\cap
D_m=\emptyset$.
\end{proof}

It is worth comparing the above result with Proposition 2.3 of
\cite{dD}, which states that every quasi-developable space is
$d$-separable. Note also that the Michael line is a quasi-developable space
(see \cite{bennett}) that is not $e$-separable.

The proof of Proposition \ref{developD} suggests that there might be a connection between $D$-spaces and $e$-separable spaces; recall that a space $X$ is a $D$-space iff for every open neighbourhood assignment $N:X\to \tau$ there is a closed discrete $D\subseteq X$ so that $N''D$ covers $X$. However, we note that the Alexandrov double circle is hereditarily $D$ (see e.g. \cite[Proposition 2.5]{gruen}) but not $e$-separable.\\

For later reference, we would like to state two results on the existence of closed discrete sets in products.

\begin{thm}[\L o\'s \cite{los}, Gorelic \cite{gorelic}]\label{gorelic} $D(\omega)^{2^\kappa}$ contains a closed discrete set of size $\kappa$ for every $\kappa$ less than the first measurable cardinal. 
\end{thm}

The above result was first proved by \L o\'s \cite{los} but the reference \cite{gorelic} is more accessible.

\begin{thm}[Mycielski \cite{myc}]\label{myc} $D(\omega)^{\kappa}$ contains a closed discrete set of size $\kappa$ for every $\kappa$ less than the first weakly inaccessible cardinal. 
\end{thm}

\section{Density and extent for $e$-separable spaces}\label{d_and_e}

Our goal now is to elaborate further on the observation that if $X$ is $e$-separable then $d(X)\leq e(X)$. In particular, in what context is the implication reversible? 

First, note that $d(X)\leq e(X)$  does not imply that there are closed discrete sets of size $d(X)$:

\begin{example} There is a $\sigma$-closed-discrete (hence $e$-separable) space $X$ which contains no closed discrete sets of size $d(X)$.
\end{example}
\begin{proof} Let $X=\oo_\oo+1$ and declare all points in $\oo_\oo$ isolated and let $\{\{\oo_\oo\}\cup A:A\in [\oo_\oo]^{<\aleph_\oo}\}$ form a neighbourhood base at $\oo_\oo$. 
\end{proof}

Next, we show that even a significant strengthening of $d(X)\leq e(X)$ fails to imply $e$-separability in general:


\begin{example} There is a 0-dimensional space $X$ such that $|X|=\omega_1$, every somewhere dense subset of $X$ contains a closed discrete subset of size $\omega_1$, while $X$ is not $e$-separable.
\end{example}
\begin{proof} Let $X=\omega_1\mbox{}^{<\omega}$ and declare $U\subseteq X$ to be open iff $x\in U$ implies that $\{\alpha<\omega_1:x\smf (\alpha)\in U\}$ contains a club. Now, $X$ is a Hausdorff, 0-dimensional and dense-in-itself space.

\begin{obs}
 A set $E\subseteq X$ is closed discrete iff $\{\alpha<\omega_1:x\smf \alpha\in E\}$ is non-stationary for every $x\in X$.
\end{obs}


This observation immediately implies that the $\sigma$-closed-discrete sets are closed discrete and hence $X$ cannot be $e$-separable.

Suppose that $Y\subseteq X$ is dense in a non-empty open set $V$; $I_x=\{\alpha\in \omg: x\smf \alpha\in Y\}$ must be stationary for any $x\in V$ and so we can select an uncountable but non-stationary $I\subseteq I_x$. Hence $\{x\smf \alpha:\alpha \in I\}$ is an uncountable closed discrete subset of $Y$.
\end{proof}

Now, let us turn to powers of a fixed space $X$. Could it be that $d(X^\kappa)\leq e(X^\kappa)$ implies that $X^\kappa$ is $e$-separable whenever $\kappa$ is an infinite cardinal? The answer is negative, at least under the assumption that there are measurable cardinals:

\begin{example}\label{ex:mble} If $\kappa$ is the first measurable cardinal, then $d(\omega^\kappa)= e(\omega^\kappa)$; however, $\omega^\kappa$ is not $e$-separable.
\end{example}
\begin{proof} It is clear that $d(\omega^\kappa)=\kappa$; also, $2^\lambda<\kappa$ whenever $\lambda<\kappa$, and so Theorem \ref{gorelic} implies that $e(\omega^\kappa)=\kappa$ as well.

If we show that $\omega^\kappa$ has no closed discrete sets of size $\kappa$ then $\omega^\kappa$ cannot be $e$-separable. Suppose that $A=\{x_\alpha:\alpha<\kappa\}\subseteq \omega^\kappa$ and that $\mc U$ is a $\sigma$-complete non-principal ultrafilter on $\kappa$. Note that $$\kappa=\bigcup_{n\in \omega}\{\alpha<\kappa:x_\alpha(\xi)=n\}$$ for each $\xi<\kappa$. So there is a unique $n\in \omega$ such that $\{\alpha<\kappa:x_\alpha(\xi)=n\}\in \mc U$. In turn, we can define $y\in \omega^\kappa$ by $y(\xi)=n$ iff $\{\alpha<\kappa:x_\alpha(\xi)=n\}\in \mc U$. It is easy to see that $\{\alpha<\kappa:x_\alpha\in V\}\in \mc U$ for every open neighbourhood $V$ of $y$, and so $V\cap A$ has size $\kappa$. Hence, $y$ is an accumulation point of $A$.
\end{proof}

However, if we suppose a bit more than $d(X^\kappa)\leq e(X^\kappa)$ then we get

\begin{thm} \label{prodthm}Let $X$ be any space and $\kappa$ an infinite cardinal. If $X^\kappa$ contains a closed discrete set of size $d(X^\kappa)$ then $X^\kappa$ is $e$-separable.
\end{thm}

The above theorem is an analogue of \cite[Theorem 1]{juhi}: if $X^\kappa$ has a discrete subspace of size $d(X)$ then $X^\kappa$ is $d$-separable. Example \ref{ex:mble} shows that assuming  ``$X^\kappa$ contains a closed discrete set of size $d(X)$'' does not imply that $X$ is $e$-separable.

We will prove a somewhat technical lemma now which immediately implies Theorem \ref{prodthm} and will be of use later as well:

\begin{lemma}\label{power0} Let $X$ be any space and $\kappa$ an infinite cardinal. Suppose that $D\subseteq X^\kappa$ is dense in $X^\kappa$ and $X^\kappa$ contains a closed discrete set of size $|D|$. Then there is a  dense set $E$ in $X^\kappa$ such that 
\begin{enumerate}
	\item $|D|=|E|$, $d(D)=d(E)$, and
	\item $E$ is $\sigma$-closed-discrete.
\end{enumerate}
\end{lemma}
\begin{proof}
Pick a countable increasing sequence $(I_n)_{n\in\omega}$ of subsets of $\kappa$ such that $\kappa=|I_n|=|\kappa\setminus I_n|$ for each $n\in\omega$ and $\kappa=\bigcup_{n\in\omega}I_n$. Fix closed discrete sets $E_n$ of size $|D|$ in $X^{\kappa\setminus I_n}$ and  bijections $\varphi_n:D\rightarrow E_n$ for each $n\in \omega$. 

We define maps $\psi_n:D\to X^\kappa$ by 

 \[
    \psi_n(d)(\xi) = \begin{cases}
		     d(\xi), & \text{for } \xi\in I_n, and\\
        \varphi_n(d)(\xi), & \text{for } \xi\in\kappa\setm I_n.
   
        \end{cases}
  \]


 Let $E=\bigcup_{n\in\omega}\ran(\psi_n)$. 
  Clearly $|D|=|E|$ holds.
 
  It is easy to see that $E$ is dense in $X^\kappa$: if $[\varepsilon]$ is a basic open set in $X^\kappa$ then there is an $n\in \omega$ such that $\dom(\varepsilon)\subseteq I_n$, hence $\ran(\psi_n)\cap [\varepsilon]\neq\emptyset$. Next we show (2) by proving that $\ran(\psi_n)$ is closed discrete as well for each $n\in\omega$. Pick any $x\in X^\kappa$. There is a basic open set $[\varepsilon]$ in $X^{\kappa\setminus I_n}$ such that $x\upharpoonright _{\kappa\setminus I_n} \in [\varepsilon]$ and $|[\varepsilon] \cap E_n|\leq 1$.
  Thus the basic open set $\{y\in X^\kappa:\varepsilon\subseteq y\}$ of $X^\kappa$, which we (by abuse of notation) also denote by $[\varepsilon]$, satisfies $x\in [\varepsilon]$ and $|[\varepsilon] \cap \ran(\psi_n)|\leq 1$.
 
 Finally we prove $d(D)=d(E)$. Note that if $D_0$ is dense in $D$ then $\bigcup_{n\in\omega}\psi_n\mbox{}''D_0$ is dense in $E$, hence $d(E)\leq d(D)$. Suppose that $A\in [E]^{<d(D)}$; we want to prove that $A$ is not dense in $E$. If $A$ is finite, there is nothing to prove. If $A$ is infinite, let $$D_A=\bigcup_{n\in\omega}\{d\in D: \psi_n(d)\in A\};$$ then $D_A$ cannot be dense in $D$ as $|D_A|\leq |A|<d(D)$. 

Fix a basic open set $U=[\varepsilon]$ such that $[\varepsilon]\cap D_A =\emptyset$. There is an $n^*\in\omega$ such that $\dom(\varepsilon) \subseteq I_{n^*}$. 
 \begin{clm} If $m\geq n^*$ then $[\varepsilon]\cap \{\psi_m(d):d\in D_A\}=\emptyset$.
 \end{clm}
 \begin{proof} Suppose that $m\geq n^*$ and $d\in D_A$. Then $d\upharpoonright I_m=\psi_m(d) \upharpoonright I_m$, $d\notin [\varepsilon]$ and $\dom(\varepsilon)\subseteq I_m$, thus $\psi_m(d)\notin [\varepsilon]$.
 \end{proof}
Hence $$U\cap \bigcup_{n\in\omega}{\psi_n(D_A)}\subseteq \bigcup_{n<n^*}{\psi_n(D_A)},$$ that is,  $U\cap \bigcup_{n\in\omega}{\psi_n(D_A)}$ is closed discrete as each ${\psi_n(D_A)}$ is closed discrete. However, $A\subseteq \bigcup_{n\in\omega}{\psi_n(D_A)}$ which shows that $A\cap U$ cannot be dense in $U$.
\end{proof}

Let us present two corollaries. The aforementioned \cite[Theorem 1]{juhi} implies that $X^{\kappa}$ is always $d$-separable for any $\kappa\ge d(X)$. We know that, say, $[0,1]^\kappa$ is not $e$-separable when $\kappa\ge \mathfrak c^+$ because of Observation \ref{cardobs}; indeed, $[0,1]^{\kappa}$ is compact so $\sigma$-closed-discrete sets are countable, but $[0,1]^{\kappa}$ is not separable. However, the following holds:

\begin{corol}\label{prodcor}
  Suppose that $X$ is not countably compact. Then $X^{d(X)}$ is $e$-separable if $d(X)$ is less than the first weakly inaccessible cardinal. 
\end{corol}
\begin{proof}
 Let $\kappa=d(X)$ and note that it suffices to find a closed discrete subspace of $X^\kappa$ of size $d(X^\kappa)$ by Theorem \ref{prodthm}. First, note that $d(X^\kappa)=\kappa$. Second, $X$ contains an infinite closed discrete subspace $Y$ since $X$ is not countably compact. So $Y^\kappa$ is a closed copy of $D(\omega)^\kappa$ in $X^\kappa$. Finally,  $D(\omega)^\kappa$ does contain a closed discrete set of size $\kappa$ by Theorem \ref{myc}.
\end{proof}

\begin{corol}
  Suppose that $X$ is not countably compact. Then $X^{2^{d(X)}}$ is $e$-separable if $d(X)$ is less than the first measurable cardinal. 
\end{corol}
\begin{proof} The proof is the same as for Corollary \ref{prodcor} but now applying Theorem \ref{gorelic}.
\end{proof}

Interestingly, if $X$ is compact Hausdorff then $X^\omega$ is $d$-separable already (see \cite[Corollary 5]{juhi}). Furthermore, Moore \cite{moore2} showed that there is an $L$-space $X$ such that $X^2$ is $d$-separable. Note that $X$ itself is not $d$-separable since each discrete subspace of $X$ is countable but $X$ has uncountable density. Moore's example was improved by Peng \cite{peng}: there is an  $L$-space $X$ such that $X^2$ is $e$-separable. We wonder if the following related question is true:

\begin{prob}
   Is there a non-separable, countably compact $X$ so that $X^2$ is $e$-separable?
\end{prob}

\section{The sizes of $\sigma$-discrete dense sets}\label{sizes}

Next, we investigate the size of the smallest $\sigma$-discrete dense set in $e$-separable spaces.

\begin{defin}
For an $e$-separable space $X$, we define $$d_e(X)=\min \{|E|:\text{ E is a dense } \sigma \text{-closed-discrete subset of } X\}.$$
\end{defin}

Clearly $d(X)\leq d_e(X)\leq e(X)$ for any $e$-separable space $X$ and next we show that $d(X)= d_e(X)$ fails to hold in general:

\begin{thm}\label{examp} There is a 0-dimensional $e$-separable space $X$ such that $$\mathfrak{c}=d(X)< d_e(X)=e(X)=w(X)=2^\mathfrak{c}.$$
\end{thm}
\begin{proof} First note the following:
\begin{clm} \label{4.3} Suppose that a space $X$ can be written as $D\cup E$ so that
\begin{enumerate}
	\item $D$ is dense in $X$,
	\item $E$ is dense and $\sigma$-closed-discrete in $X$,
	\item $d(D)<d(E)$, and
	\item every $A\in [D]^{\leq e(D)}$ is nowhere dense in $X$ (or equivalently, in $D$).
\end{enumerate}
 Then $X$ is $e$-separable and $d(X)< d_e(X)$.
\end{clm}
\begin{proof}
X is $e$-separable by (2) and $d(X)\leq d(D)$ by (1). We prove that if $F\in [X]^{\leq d(X)}$ and $F$ is $\sigma$-closed-discrete then $F$ is not dense in $X$; this proves the claim. Take $F\subseteq X$ as above and note that by (3) there is a non-empty open set $U\subseteq X$ such that $U\cap E \cap F = \emptyset$. As $|F\cap D|\leq e(D)$, $F\cap D$ must be nowhere dense in $X$. Thus there is a non-empty open $V\subseteq U$ such that $V \cap F \cap D =\emptyset$. Thus $V \cap F =\emptyset$ showing that $F$ is not dense.
\end{proof}

Now, it suffices to construct a 0-dimensional space $X=D\cup E$ satisfying (1)-(4). Let us construct $X=D\cup E\subseteq \omega ^{2^\mathfrak{c}}$ such that
\begin{enumerate}[(i)]
	\item $D$ is dense in $\omega ^{2^\mathfrak{c}}$,
	\item $E$ is dense and $\sigma$-closed-discrete in $\omega ^{2^\mathfrak{c}}$, 
	\item $|D|=\mathfrak{c}$ and $d(E)=2^{\mathfrak{c}}$, and
	\item $e(D)=\omega$.
\end{enumerate}
It is trivial to see that (i)-(iii) implies (1)-(3), respectively, while (iv) implies (4) using the fact that $d(\omega ^{2^\mathfrak{c}})=\mathfrak{c}$.

First we construct $D$. Construct dense subsets $D_n\subseteq n^{2^\mathfrak{c}}$ of size $\mathfrak{c}$ which are countably compact, for each $n\in\omega$; this can be done by choosing a dense subset $D^0_n\subseteq n^{2^\mathfrak{c}}$ of size $\mathfrak{c}$ and adding accumulation points recursively ($\omega_1$ many times) for all countable subsets. Define $D=\bigcup_{n\in\omega}D_n$. Then $D$ is dense in $\omega^{2^\mathfrak{c}}$ as $\bigcup_{n\in\omega} n^{2^\mathfrak{c}}$ is dense in $\omega^{2^\mathfrak{c}}$ and $e(D)=\omega$ as $e(D_n)=\omega$ for all $n\in\omega$; thus $D$ satisfies (i), (iv) and the first part of (iii).

Now, we construct $E$ satisfying (ii) and (iii) which finishes the proof. Let $S=\sigma(\omega^{2^\mathfrak{c}})=\{x\in\omega^{2^\mathfrak{c}}:|\{\alpha\in 2^\mathfrak{c}:x(\alpha)\neq 0\}|<\aleph_0\}$; then $d(S)=2^\mathfrak{c}$ and $S$ is dense in $\omega^{2^\mathfrak{c}}$. Recall that $\omega^{2^\mathfrak{c}}$ contains a closed discrete set of size $2^\mathfrak{c}$ by Theorem \ref{myc}. Now, by applying Lemma \ref{power0}, we find a $\sigma$-closed-discrete $E$ which is dense in $\omega^{2^\mathfrak{c}}$ and satisfies $d(E)=d(S)=2^\mathfrak{c}$.
\end{proof}

 Naturally, one can consider the same problem for $d$-separable spaces. Let us present an example along the same lines under the assumption $\aleph_1<\mathfrak c$:

\begin{prop}\label{prop:dsep} Suppose that  $\aleph_1<\mathfrak c$. Then there is a $d$-separable space $X$ with $d(X)=\aleph_1$ that contains no dense $\sigma$-discrete sets of size $\aleph_1$.
\end{prop}

\begin{proof} J. Moore \cite[Theorem 5.4]{moore} proved that there is a colouring $c:[\oo_1]^2\to \oo$ such that for every $n\in \oo$, uncountable pairwise disjoint $A\subseteq [\oo_1]^n$, uncountable $B\subseteq \oo_1$ and $h:n\to \oo$ there exist $a\in A$ and $\beta\in B\setm \max(a)$ such that $c(a(i),\beta)=h(i)$ for every $i<n$, where $a=\{a(i):i<n\}$.

Suppose that $D=\{d_n:n\in\omega\}$ is any countable space.

\begin{clm} \label{jm} There is a dense and hereditarily Lindel\"of subspace $Y\subseteq D^{\oo_1}$ that is not separable.
\end{clm}
\begin{proof}
For each $\beta<\omg$, define $y_{\beta}\in  D^\omg$ as follows:
\begin{equation}\label{eq:ya}
y_{\beta}({\alpha})=\left \{
\begin{array}{ll}
d_{c(\alpha,\beta)}&\text{if ${\alpha}<{\beta}$,}\\ 
d_0&\text{if ${\alpha}\geq {\beta}$.}
\end{array}
\right . 
\end{equation}
Now let $Y=\{y_{\beta}:\beta<\oo_1\}$.

We claim that there is an $\alpha<\omg$ so that $Y\uhp (\omg\setm \alpha)$ is dense in $D^{\omg\setm \alpha}\simeq D^\omg$. Suppose otherwise: then we can find basic open sets $[\varepsilon_\alpha]$ in $D^{\omg\setm \alpha}$ so that $Y\cap [\varepsilon_\alpha]=\emptyset$. By standard $\Delta$-system arguments, we find $I\in [\omg]^{\aleph_1}$, $n\in \omega$ and $h:n\to \omega$ so that $\dom(\varepsilon_\alpha)=\{a_\alpha(i):i<n\}$ are pairwise disjoint for $\alpha\in I$ and $d_{h(i)}\in \varepsilon_\alpha(a_\alpha(i))$ for each $i<n$. Now, there exist $\alpha\in I$ and $\beta\in \omg\setm \max(\dom(\varepsilon_\alpha))$ so that $c(a_\alpha(i),\beta)=h(i)$ for all $i<n$. This means that $d_{c(a_\alpha(i),\beta)}\in \varepsilon_\alpha(a_\alpha(i))$ for $i<n$ and so $y_\beta\in [\varepsilon_\alpha]$. This contradicts our assumption.

It is clear that $d(Y\uhp (\omg\setm \alpha))=\aleph_1$. It remains to prove that $Y\uhp (\omg\setm \alpha)$ is hereditarily Lindel\"of.

Fix $W\in[\omega_1]^{\aleph_1}$ and, for each $\gamma\in W$, let $[\varepsilon_\gamma]$ be a basic open subset of $D^{\omega_1\setminus\alpha}$ with $y_\gamma\uhp(\omg\setminus\alpha)\in[\varepsilon_\gamma]$; we may assume that $\max(\dom(\varepsilon_\gamma))>\gamma$.
Suppose, by way of contradiction, that for each $\eta<\omg$
we have $\{y_\gamma\uhp(\omg\setminus\alpha):\gamma\in W\}\nsubseteq\bigcup\{[\varepsilon_\gamma]\colon\gamma\in W\cap\eta\}$.
We can then recursively define, for $\zeta<\omg$,

$\cdot$  $\delta_0$ as the least element of $W$;

$\cdot$  $\delta_{\zeta+1}$ as the least $\delta\in W$ satisfying $\delta>\sup_{\eta\le\zeta}\max(\dom(\varepsilon_{\delta_\eta}))$ and $y_\delta\uhp(\omg\setminus\alpha)\notin\bigcup\{[\varepsilon_\gamma]\colon\gamma\in W\cap(\delta_\zeta+1)\}$;

$\cdot$  $\delta_\zeta$ as the least $\delta\in W$ satisfying $\delta>\sup_{\eta<\zeta}\max(\dom(\varepsilon_{\delta_\eta}))$ and $y_\delta\uhp(\omg\setminus\alpha)\notin\bigcup\{[\varepsilon_\gamma]\colon\gamma\in W\cap\sup_{\eta<\zeta}\delta_\eta\}$
if $\zeta$ is a limit ordinal.

Again by $\Delta$-system arguments, there exist
$r\in[\omg\setminus\alpha]^{<\aleph_0}$,
$p:r\to D$,
$Z\in[\omg]^{\aleph_1}$,
$n\in\omega$
and
$h:n\to\omega$
satisfying
\begin{itemize}
\item[$(i)$]
  $r\subseteq\dom(\varepsilon_{\delta_\zeta})$ for all $\zeta\in Z$;
\item[$(ii)$]
  $\dom(\varepsilon_{\delta_\zeta})\setminus r=\{a_\zeta(i):i<n\}$ are pairwise disjoint for $\zeta\in Z$;
\item[$(iii)$]
  $d_{h(i)}\in \varepsilon_{\delta_\zeta}(a_\zeta(i))$ for each $i<n$; and
\item[$(iv)$]
  $p\subseteq y_{\delta_\zeta}$ for all $\zeta\in Z$.
\end{itemize}
Now, there are $\zeta,\zeta'\in Z$ such that $\delta_{\zeta'}\ge\max(\dom(\varepsilon_{\delta_\zeta}))$ and $c(a_\zeta(i),\delta_{\zeta'})=h(i)$ for all $i<n$.
Thus $d_{c(a_\zeta(i),\delta_{\zeta'})}\in \varepsilon_{\delta_\zeta}(a_\zeta(i))$ for $i<n$, whence
$y_{\delta_{\zeta'}}\uhp(\omg\setminus\alpha) \in [\varepsilon_{\delta_\zeta}]$
-- which is a contradiction since the fact that $\delta_{\zeta'}\ge\max(\dom(\varepsilon_{\delta_\zeta}))>\delta_\zeta$
implies $\zeta<\zeta'$ by construction.
\end{proof}

Now, by $\mathfrak c \geq \aleph_2$, we can pick a countable dense $D\subseteq\omega^{\omega_2}$. Then $D^{\omega_1}$ is dense in $(\omega^{\omega_2})^{\omega_1} \simeq \omega^{\omega_2}$. By Claim \ref{jm}, there is a dense $Y \subseteq D^{\omega_1}$ such that every discrete subset of $Y$ is countable and hence nowhere dense (as $Y$ is non-separable). Now, by Lemma \ref{power0}, there is a dense $\sigma$-closed-discrete $E \subseteq \omega^{\omega_2}$ satisfying $d(E) = \aleph_2$ -- in view of Theorem \ref{myc} and the fact that e.g. $\sigma(\omega^{\omega_2})=\{x\in\omega^{\omega_2}:|\{\alpha\in\omega_2:x(\alpha)\neq 0\}|<\aleph_0\}$ is a dense subset of $\omega^{\omega_2}$ with density $\aleph_2$.

Let $X=Y\cup E$. An argument strictly analogue to what is done in Claim \ref{4.3} finishes the proof. 
\end{proof}

The assumption $\aleph_1<\mathfrak c$ is somewhat unnatural in Proposition \ref{prop:dsep} but we do not know how to remove it:

\begin{prob}
 Is there a ZFC example of a $d$-separable space $X$ with the property that every $\sigma$-discrete dense subset of $X$ has cardinality greater than $d(X)$?
\end{prob}

In particular, we cannot answer the following:

\begin{prob}
     Is there, in ZFC, a dense $Y\subseteq 2^{\oo_2}$ of size $\aleph_1$ all of whose $\sigma$-discrete subsets are nowhere dense?
\end{prob}

Finally, recall that any compact, $e$-separable space satisfies $d(X)=d_e(X)$. We do not know if the analogue holds for $d$-separable spaces:

\begin{prob}Is there a $\sigma$-discrete dense subset of size $d(X)$ in any compact, $d$-separable space $X$?
\end{prob}

\section{Preservation under products}\label{pressec}

As mentioned in the introduction, the behaviour of separable and $d$-separable spaces under products and powers is very well described: separability is preserved by products of size $\leq \mathfrak c$ but not bigger; on the other hand, the product of $d$-separable spaces is always $d$-separable. Hence our goal in this section is answering the following natural question: for which cardinals $\kappa$ is it true that every product of $\kappa$ many $e$-separable spaces is $e$-separable? As noted earlier in Example \ref{2^c+}, any such $\kappa$ is at most the continuum.

Let us start with powers of a single $e$-separable space.
We would like to thank Ofelia T. Alas for pointing out the following
to us:

\begin{prop}[Alas]
\label{cpower}

Let $X$ be an $e$-separable space and $\kappa\le\mathfrak{c}$. Then
the space $X^\kappa$ is $e$-separable.

\end{prop}

\begin{proof}

Let $(D_n)_{n\in\omega}$ be a sequence of closed discrete subsets of
$X$ with $\bigcup_{n\in\omega}D_n$ dense in $X$.
Fix a subspace $Y\subseteq\mathbb{R}$ with $|Y|=\kappa$, and let
$\mathcal{B}$ be a countable base for $Y$. Now consider
$T=\bigcup_{n\in\omega}(S_n\times\mbox{}^n\omega)$, where
$S_n=\{(B_0,\dots,B_{n-1})\in\mbox{}^n\mathcal{B}:\forall
i,j<n\;(i\neq j\Rightarrow B_i\cap B_j\neq\emptyset)\}$ for every
$n\in\omega$.

Fix an arbitrary $p\in X$.
For each $t=((B_0,\dots,B_{n-1}),(k_0,\dots,k_{n-1}))\in T$, we
define $E_t$ to be the set of those $x\in X^Y$ so that there is an $(a_i)_{i<n}\in\prod\{D_{k_i}:i<n\}$ with 

 \[
   x(\alpha) = \begin{cases}
		      a_i, & \text{for } \alpha\in B_i \text { and }i<n, \text{ and}\\  
p, & \text{for } \alpha\in Y\setminus\bigcup_{i=0}^{n-1}B_i.
   
        \end{cases}
  \]

It is routine to verify that each $E_t$ is a closed discrete subspace
of $X^Y$ and that $\bigcup_{t\in T}E_t$ is dense in $X^Y$. Since $T$
is countable and $|Y|=\kappa$, it follows that $X^\kappa$ is
$e$-separable.
\end{proof}

Now, we turn to arbitrary products of $e$-separable spaces. We will see that the heart of the matter is whether we can find \emph{large closed discrete sets in the product of small discrete spaces}.

In \cite{Mrow}, Mr\'owka introduced a class of cardinals denoted by $\mc M^*$: we write $\lambda\in \mc M^*$ iff  there is a product of $\lambda$ many discrete spaces $X=\prod \{X_\alpha:\alpha<\lambda\}$ each of size $<\lambda$ so that $X$ has a closed discrete set of size $\lambda$. Equivalently, the product $\prod\{D(\nu)^\lambda:\nu\in \lambda\cap \card\}$ contains a closed discrete set of size $\lambda$.

If a cardinal $\lambda$ is in $\mathcal M^*$ then some degree of compactess fails for $\lambda$. Let us make this statement precise: recall that $\mc L_{\lambda,\omega}$ is the infinitary language which allows conjunctions and disjunctions of  $<\lambda$ formulas and universal or existential quantification over finitely many variables. The language $\mc L_{\lambda,\oo}$ is \emph{weakly compact} by definition if every set of at most $\lambda$ sentences $\Sigma$ from $\mc L_{\lambda,\oo}$ has a model provided that every $S\in [\Sigma]^{<\lambda}$ has a model (see \cite{Jech}, p. 382).

\begin{thm}[Mr\'owka \cite{Mrow}, Chudnovsky \cite{Cud}] $\lambda \notin \mc M^*$ if and only if  $\mc L_{\lambda,\oo}$ is weakly compact.
\end{thm}

Now, as expected, $\lambda \notin \mc M^*$ -- or, equivalently, the statement ``$\mc L_{\lambda,\oo}$ is weakly compact'' -- has some large cardinal strength. First, we mention two classical results:

\begin{lemma}\cite[Exercises 17.17 and 17.18]{Jech} \label{jechlemma} If $\mc L_{\lambda,\oo}$ is weakly compact then $\lambda$ is weakly inaccessible. 
\end{lemma}

\begin{lemma}\cite[Theorem 17.13]{Jech} $\lambda$ is a weakly compact cardinal iff it is strongly inaccessible and $\mc L_{\lambda,\oo}$ is weakly compact.
\end{lemma}

For our current purposes, one can consider the above lemma the definition of weakly compact cardinals. Now, given a  weakly compact cardinal $\lambda$, one can enlarge the continuum while the language $\mc L_{\lambda,\oo}$ remains weakly compact:

\begin{thm}[Chudnovsky \cite{Cud}, Boos \cite{Boos}]\label{cud} If $\lambda$ is a weakly compact cardinal and $\mb C_{\lambda^+}$ is the poset for adding $\lambda^+$ many Cohen-reals then $V^{\mb C_{\lambda^+}}\models$ ``$\mc L_{\lambda,\omega}$ is weakly compact hence  $\mathfrak{c}\setm \mc M^*\neq \emptyset$''.
\end{thm}

Finally, recently B. Cody, S. Cox, J. D. Hamkins and T. Johnstone \cite{joel, joel2} showed that a weakly compact cardinal can be recovered from $\mc L_{\lambda,\oo}$ being weakly compact:

\begin{thm}\label{joel}
 If $\mc L_{\lambda,\oo}$ is weakly compact then $\lambda$ is weakly compact in $L$.
\end{thm}

Now, it is easy to derive our first main result about non-preservation:

\begin{lemma}\label{main1}
 If $\lambda\leq \mathfrak c$ and $\lambda\notin \mc M^*$ then there is a non-$e$-separable product of $\lambda$ many discrete spaces.
\end{lemma}
\begin{proof}
 $\lambda\notin \mc M^*$ implies that $\mc L_{\lambda,\omega}$ is weakly compact and hence $\lambda$ is a regular limit cardinal. Now take discrete spaces $X_\alpha$ of size $<\lambda$ such that $\sup \{|X_\alpha|:\alpha<\lambda\}=\lambda$. The product $X=\prod \{X_\alpha:\alpha<\lambda\}$ contains no closed discrete subsets of size $\lambda$ as  $\lambda\notin \mc M^*$. We claim that $d(X)=\lambda$, which follows from the following more general observation:

\begin{obs}\label{dens} Suppose that $\kappa \leq \mathfrak{c}$ and $X_\alpha$ is discrete for $\alpha<\kappa$. Then $d(\prod \{X_\alpha:\alpha<\kappa\})=\sup \{|X_\alpha|:\alpha<\kappa\}$.
\end{obs}

To prove this observation, simply apply the usual trick appearing in the proof of Proposition \ref{cpower}.

Now, we claim that $X$ cannot be $e$-separable. Indeed, if $X$ is $e$-separable then Observation \ref{cardobs} implies that $X$ has a closed discrete subset of size $d(X)=\lambda=cf(\lambda)>\omega$; however, this is not the case.
\end{proof}

Hence, we immediately get the following:

\begin{corol}\label{main1cor} If the existence of a weakly compact cardinal is consistent with ZFC then so is the statement that there is a non-$e$-separable product of less than $\mathfrak{c}$ many discrete spaces.
\end{corol}
\begin{proof}
 Apply Lemma \ref{main1} and  Theorem \ref{cud}.
\end{proof}

Now, we will obtain that it is also consistent with ZFC that every product of at most $\mathfrak c$ many $e$-separable spaces is $e$-separable; we will do so by showing that this last statement is implied by the non-existence of weakly compact cardinals in $L$. It will suffice to prove

\begin{thm}\label{main2} Suppose that $\lambda\leq \mathfrak{c}$ is minimal so that there is a family of $\lambda$ many $e$-separable spaces with non-$e$-separable product. Then $\lambda\notin \mc M^*$ and so $\mc L_{\lambda,\omega}$ is weakly compact. 
\end{thm}

Let us mention that $\mc L_{\mathfrak c,\omega}$ is not weakly compact \cite{joel} and so $\lambda<\mathfrak c$ in the previous theorem. In any case, if $\mc L_{\lambda,\omega}$ is weakly compact then $\lambda$ is weakly compact in $L$ by Theorem \ref{joel}. In turn, we have the following result:

\begin{corol}\label{main2cor} If there is a non-$e$-separable product of at most $\mathfrak c$ many $e$-separable spaces then there is a weakly compact cardinal in $L$.
\end{corol}

By combining Corollaries \ref{main1cor} and \ref{main2cor}, we obtain:

\begin{corol}\label{maincor}
  The following statements are equiconsistent relative to ZFC:
  \begin{itemize}
  \item[$(a)$]
    there is a product of at most $\mathfrak c$ many $e$-separable spaces that fails to be $e$-separable;
  \item[$(b)$]
    there is a weakly compact cardinal.
  \end{itemize}
\end{corol}

Let us now turn to proving Theorem \ref{main2}. First, we start by reducing the problem to products of discrete spaces again:

\begin{lemma}
\label{lemmahmp}

Suppose that $\kappa\le\mathfrak{c}$. Then the following are equivalent:

\begin{itemize}

\item[$(a)$]
every product of at most $\kappa$ many $e$-separable spaces is
$e$-separable;

\item[$(b)$]
every product of at most $\kappa$ many discrete spaces is
$e$-separable.

\end{itemize}

\end{lemma}

\begin{proof}

The implication $(a)\Rightarrow(b)$ holds trivially. We prove
$(b)\Rightarrow(a)$.

Let $X=\prod\{X_\alpha:\alpha\in Y\}$, where $Y\subseteq\mathbb{R}$ has
cardinality at most $\kappa$ and each $X_\alpha$ is $e$-separable. For
each $\alpha\in Y$, fix a point $p_\alpha\in X_\alpha$ and a sequence
$(E^\alpha_k)_{k\in\omega}$ of closed discrete subsets of $X_\alpha$
with $\overline{\bigcup_{k\in\omega}E^\alpha_k}=X_\alpha$.

Fix a countable base $\mathcal{B}$ for $Y$ and, for each
$n\in\omega$, consider
$$
S_n=\{(B_i)_{i<n}\in\mbox{}^n\mathcal{B}:\forall
i,j<n\;(i\neq j\Rightarrow B_i\cap B_j=\emptyset)\};
$$
now, for each $t=((B_i)_{i<n},(k_0,\dots,k_{n-1}))\in S_n\times\mbox{}^n\omega$, define
$Y_t$ to be the set of those $x\in X$ so that 
 \[
   x(\alpha) = \begin{cases}
		      x'_\alpha & \text{ for some } x'_\alpha\in E^\alpha_{k_i} \text{ for } \alpha\in B_i \text { and }i<n, \text{ and}\\  
p_\alpha, & \text{for } \alpha\in Y\setminus\bigcup_{i=0}^{n-1}B_i.
   
        \end{cases}
  \]
Note that each $Y_t$ is homeomorphic to the product
$\prod_{i<n}\prod_{\alpha\in B_i}E^\alpha_{k_i}$. Hence $Y_t$ is is
$e$-separable by $(b)$. Let $(D^t_k)_{k\in\omega}$ be a sequence
of closed discrete subsets of $Y_t$ with
$\overline{\bigcup_{k\in\omega}D^t_k}=Y_t$.
Since each $Y_t$ is closed in $X$, we have that each $D^t_k$ is a
closed discrete subset of $X$. Finally, as
$\bigcup_{n\in\omega}\bigcup_{r\in S_n\times\mbox{}^n\omega}Y_t$ is 
dense in $X$, it follows that
$$
\overline{\bigcup_{n\in\omega}\bigcup_{t\in
S_n\times\mbox{}^n\omega}\bigcup_{k\in\omega}D^t_k}=X,
$$
thus showing that $X$ is $e$-separable.
\end{proof}

Note that we immediately get the following easy:

\begin{corol}\label{finiteprod}
 The product of finitely many $e$-separable spaces is $e$-separable.
\end{corol}

Second, we show that as long as we take the product of large discrete sets relative to the number of terms, we end up with an $e$-separable product:

\begin{lemma}
\label{reg}

Let $\kappa$ be an infinite cardinal. Then the product of
at most
$\kappa$ many discrete spaces of cardinality at least
$\kappa$ is $e$-separable.

\end{lemma}

\begin{proof}

Let $X=\prod\{X_\alpha:\alpha\in\lambda\}$, where $\lambda\le\kappa$
and each $X_\alpha$ is a discrete space with cardinality at least
$\kappa$. We can assume that $\lambda$ is infinite and that
$X_\alpha=|X_\alpha|$ for all $\alpha\in\lambda$. 

Define
$$
P^i_j=\{(F,p)\in[\lambda]^i\times
Fn(\lambda,\kappa):|p|=j\textrm{ and }F\cap\mathrm{dom}(p)=\emptyset\}
$$
for each $i,j\in\omega$ where $Fn(\lambda,\kappa)$ denotes the set of finite partial functions from $\lambda$ to $\kappa$. Fix an injective function
$\varphi:\bigcup_{i,j\in\omega}P^i_j\rightarrow\kappa$
such that $\varphi(F,p)>\max(\mathrm{ran}(p))$ for every
$(F,p)\in\bigcup_{i,j\in\omega}P^i_j$.

Now, for every $i,j\in\omega$, let $E^i_j$ be the set of all $x\in X$
for which there is $(F,p)\in P^i_j$ satisfying
\begin{enumerate}
\item $x(\xi)\ge\kappa$ for all $\xi\in F$,
\item $x\in [p]$, and
\item $x(\xi)=\varphi(F,p)$ for all $\xi\in\lambda\setminus(F\cup\mathrm{dom}(p))$.
\end{enumerate}

It is straightforward to verify that $\bigcup_{i,j\in\omega}E^i_j$ is
dense in $X$. We claim that each $E^i_j$ is a closed discrete subset
of $X$, which will conclude our proof.

From this point on, let $i,j\in\omega$ be fixed.

To see that $E^i_j$ is discrete, pick an arbitrary $x\in E^i_j$, and
let this be witnessed by the pair $(F,p)\in P^i_j$. Note that the choice of $\varphi$ ensures that this $(F,p)$ is unique. Pick any
$\eta\in\lambda\setminus(F\cup\mathrm{dom}(p))$ and let $$V=[x\uhp (\dom(p)\cup F\cup \{\eta\})].$$ Then $V$ is an open
neighbourhood of $x$ in $X$ satisfying $E^i_j\cap V=\{x\}$.

It remains to show that $E^i_j$ is closed in $X$. Let then $y\in
X\setminus E^i_j$; we must find an open neighbourhood $V$ of $y$ in
$X$ such that $V\cap E^i_j=\emptyset$. We shall do so by considering
several cases.\\
\emph{$\cdot$ Case 1.}
$G=\{\xi\in\lambda:y(\xi)\ge\kappa\}$ has more than $i$ elements.

Then we may take any $H\in[G]^{i+1}$ and define $V=[y\uhp H]$.\\
\emph{$\cdot$ Case 2.}
$G=\{\xi\in\lambda:y(\xi)\ge\kappa\}$ has cardinality at most $i$.

We will split this case in two:\\
\emph{$\cdot$ Case 2.1.}
$\mathrm{ran}(y)\cap\kappa$ is infinite.

Then we can take $A\in[\kappa]^{j+2}$ such that
$y''A\in[\kappa]^{j+2}$ and define $V=[y\uhp A]$.\\
\emph{$\cdot$ Case 2.2.}
$\mathrm{ran}(y)\cap\kappa$ is finite.

Let $\mu=\max(\mathrm{ran}(y)\cap\kappa)$ and
$H=\{\xi\in\lambda:y(\xi)<\mu\}$.
We divide this case into three subcases:\\
\emph{$\cdot$ Case 2.2.1.}
$|H|>j$.

Pick $H'\in[H]^{j+1}$ and $\beta\in\lambda$ such that
$y(\beta)=\mu$. Now take $V=[g\uhp (H'\cup\{\beta\})]$.\\
\emph{$\cdot$ Case 2.2.2.}
$|H|\le j$ and $\mu\notin\mathrm{ran}(\varphi)$.

Let $B\in[\lambda]^{j+1-|H|}$ be such that
$y''B=\{\mu\}$ and consider $V=[g\uhp (H\cup B)]$.\\
\emph{$\cdot$ Case 2.2.3.}
$|H|\le j$ and $\mu\in\mathrm{ran}(\varphi)$.

Let $(F,p)\in P^i_j$ be such that $\varphi(F,p)=\mu$ and, as in the
previous case, take $B\in[\lambda]^{j+1-|H|}$ satisfying
$y''B=\{\mu\}$. Now define
$$V=[g\uhp (G\cup H\cup B\cup F\cup\mathrm{dom}(p))].$$
Suppose, in order to get a contradiction, that there is $x\in V\cap
E^i_j$ and let $(F',p')\in P^i_j$ witness that $x\in E^i_j$. Since
$|H\cup B|=j+1$ and $x''(H\cup B)=y''(H\cup B)\subseteq\kappa$, we have
that $\varphi(F',p')=\max(x''(H\cup B))=\max(y''(H\cup B))=\mu$. Hence
$(F',p')=(F,p)$ by injectivity of $\varphi$. Now, since
$F=\{\xi\in\lambda:x(\xi)\ge\kappa\}$ and
$G=\{\xi\in\lambda:y(\xi)\ge\kappa\}$, it follows from
$x\upharpoonright(F\cup G)=y\upharpoonright(F\cup G)$ that
$F=G$. Similarly, as $H=\{\xi\in\lambda:y(\xi)<\mu\}$ and
$\mathrm{dom}(p)=\{\xi\in\lambda:x(\xi)<\mu\}$, it follows from
$x\upharpoonright(H\cup\mathrm{dom}(p))=y\upharpoonright(H\cup\mathrm{dom}(p))$
that $H=\mathrm{dom}(p)$. Thus the pair
$(G,y\!\upharpoonright\!H)=(F,p)\in P^i_j$ witnesses that $y\in
E^i_j$, a contradiction.
\end{proof}

Finally, we are ready to present

\begin{proof}[Proof of Theorem \ref{main2}]
 Suppose that $\lambda\leq \mathfrak{c}$ is minimal so that there are $e$-separable spaces $X_\alpha$ such that $X=\prod \{X_\alpha:\alpha<\lambda\}$ is not $e$-separable. By Lemma \ref{lemmahmp}, we can suppose that each  $X_\alpha$ is discrete.

Note that $$X\simeq\prod \{X_\alpha: \alpha<\lambda, |X_\alpha|<\lambda\}\times \prod \{X_\alpha: \alpha<\lambda, |X_\alpha|\geq \lambda\}.$$ We know that the second term on the right-hand side is $e$-separable by Lemma \ref{reg}. So if $X$ is not $e$-separable then $\prod \{X_\alpha: \alpha<\lambda, |X_\alpha|<\lambda\}$ is not $e$-separable either by Corollary \ref{finiteprod}.

Now, we define $Y_\nu=\prod\{X_\alpha:\alpha<\lambda, |X_\alpha|=\nu\}$ for $\nu\in \lambda\cap \card$. Note that $Y_\nu$ is $e$-separable by Theorem \ref{cpower}. Hence, the minimality of $\lambda$ implies that $I=\{\nu\in \lambda\cap \card:Y_\nu\neq\emptyset\}$ has size $\lambda$; otherwise $X\simeq\prod\{Y_\nu:\nu\in I\}$ is a smaller non-$e$-separable product of $e$-separable spaces. Note that this already shows that $\lambda=\oo_\lambda$. 


Let us suppose that $\lambda \in \mc M^*$; we will arrive at a contradiction shortly. Take a decreasing sequence $(I_n)_{n\in\oo}$ of subsets of $I$ so that $\bigcap\{I_n:n\in\oo\}=\emptyset$ and $\lambda=|I_n|=|I\setm I_n|$ for each $n\in\oo$. Note that $d(\prod\{Y_\nu:\nu\in I\setm I_n\})=\lambda$ by Observation \ref{dens}.

\begin{clm} $\prod\{Y_\nu:\nu\in I_n\}$ contains a closed discrete set of size $\lambda$. 
\end{clm}

\begin{proof} $\lambda \in \mc M^*$ implies that $Z=\prod \{D(\nu)^\lambda : \nu\in \lambda\cap \card\}$ contains a closed discrete subset of size $\lambda$. Hence, it suffices to show that $Z$ embeds into $\prod\{Y_\nu:\nu\in I_n\}$ as a closed subspace. In order to do that, note that the set $\{\nu\in I_n: \nu>\nu_0\}$ has size $\lambda$ for every $\nu_0\in \lambda\cap \card$. Now it is routine to construct the embedding of $Z$.
\end{proof}

Finally, we can apply Lemma \ref{glue} to see that the product $X=\prod\{Y_\nu:\nu\in I\}$ must be $e$-separable. This contradicts our initial assumption on $X$.
\end{proof}

\section{Final remarks and further questions}

First, referring back to Section \ref{d_and_e}, it is natural to ask if we can say something similar to Theorem \ref{prodthm} about products. Let us present a result in this direction:


\begin{lemma}\label{glue} Suppose that $\kappa$ is an infinite cardinal and there is a decreasing sequence $(I_n)_{n\in\oo}$ of non-empty subsets of $\kappa$ with empty intersection such that $\prod\{X_\alpha:\alpha\in I_n\}$ contains a closed discrete subset of size $\delta_n=d(\prod\{X_\alpha:\alpha\in \kappa\setm I_n\})$ for every $n\in\oo$. Then $X=\prod\{X_\alpha:\alpha<\kappa\}$ is $e$-separable.
\end{lemma}
\begin{proof} Let $X(J)$ denote $\prod\{X_\alpha:\alpha\in J\}$ for $J\subseteq \kappa$. Pick $D_n=\{d^n_\xi:\xi<\delta_n\}\subseteq X(\kappa\setm I_n)$ dense and let $F_n=\{f^n_\xi:\xi<\delta_n\}\subseteq X(I_n)$ be closed discrete.

Now, for each $n\in\omega$, we define $e^n_\xi\in X$ for $\xi<\delta_n$ by
 \[
   e^n_\xi(\alpha) = \begin{cases}
		      d^n_\xi(\alpha), & \text{for } \alpha\in \kappa\setm I_n, \text{ and}\\  
f^n_\xi(\alpha), & \text{for } \alpha\in I_n.
   
        \end{cases}
  \]
We claim that the set $E_n=\{e^n_\xi:\xi<\delta_n\}$ is closed discrete. This follows from
 
\begin{obs}\label{proj} Suppose that $E\subseteq \prod\{ X_\alpha: \alpha <\kappa\}$ and there is $I\subseteq \kappa$  such that $\pi_I$ is 1-1 on $E$ and the image $\pi_I\mbox{}''E$ is closed discrete in $\prod\{ X_\alpha: \alpha \in I\}$. Then $E$ is closed discrete.
\end{obs}

Now, it is clear that $\bigcup \{E_n:n\in\oo\}$ is a dense and $\sigma$-closed-discrete subset of $X$.
\end{proof}

Second, recall that if $D(\lambda)$ is the discrete space of size $\lambda\geq \kappa$ then $D(\lambda)^\kappa$ is $e$-separable by Lemma \ref{reg}. Actually, we can say a bit more in this case:

\begin{lemma}
\label{discrprod}

Let $(\kappa_i)_{i\in I}$ be a sequence of cardinals and consider the
product space $X=\prod\{D(\kappa_i):i\in I\}$.
Suppose that the set $\{i\in I:\kappa_i=\kappa\}$ is infinite, where
$\kappa=\sum_{i\in I}\kappa_i$.
Then $X$ has a $\sigma$-discrete $\pi$-base.

\end{lemma}

\begin{proof}

Let $J$ be a countable infinite subset of $\{i\in
I:\kappa_i=\kappa\}$. Note that 
$\kappa_j=\sum_{i\in I\setminus\{j\}}\kappa_i$ for all $j\in J$. Now let $\{p^j_n(\alpha):\alpha\in\kappa_j\}$ be an
enumeration of the set $$\{p\subseteq\bigcup_{i\in
  I\setminus\{j\}}(\{i\}\times\kappa_i):p\textrm{ is a function and
}|p|=n\}$$ for every $j\in
J$ and $n\in\omega$.
Consider
$$A^j_n=\{p^j_n(\alpha)\cup\{(j,\alpha)\}:\alpha\in\kappa_j\};$$
finally, define
$\mathcal{V}^j_n=\{[q]:q\in A^j_n\}$.

Note that each $\mathcal{V}^j_n$ is a discrete family: if
$a=(a_i)_{i\in I}$ is any point of $X$ then $$U=\{(x_i)_{i\in I}\in
X:x_j=a_j\}$$ is an open neighbourhood of $a$ in $X$ such that
$$\{V\in\mathcal{V}^j_n:V\cap
U\neq\emptyset\}=\{V_{p^j_n(a_j)\cup\{(j,a_j)\}}\}.$$

Moreover, $\mathcal{V}=\bigcup_{j\in
  J}\bigcup_{n\in\omega}\mathcal{V}^j_n$ is a $\pi$-base for $X$,
since any non-empty open subset of $X$ is determined by a finite number
of coordinates which constitutes a finite subset of
$I\setminus\{j\}$ for some $j\in J$.
\end{proof}

\begin{corol}
\label{discrpower}

If $\lambda\geq \kappa$ then $D(\lambda)^\kappa$ has a $\sigma$-discrete $\pi$-base.
\end{corol}

Finally, selection principles (see e.g. \cite{scheep}) and selective versions of separability and $d$-separability (see e.g. \cite{weaksep}) were proved to be fascinating notions to study. So let us introduce the selective version of $e$-separability:

\begin{defin}
\label{E-sep}

A topological space $X$ is \emph{$E$-separable} if for every sequence of dense sets
 $(D_n)_{n\in\omega}$ of $X$ we can select  $E_n\subseteq D_n$ so that $E_n$ is closed discrete in $X$ and
 $\bigcup_{n\in\omega}E_n$ is dense in $X$.

\end{defin}

Note that every space with a $\sigma$-discrete $\pi$-base is $E$-separable as well. Let us point out that the example of Theorem \ref{examp} is an $e$-separable space which is not $E$-separable.

We ask the following questions:

\begin{prob}
 Suppose that $X$ is an $e$-separable space which is the product of discrete spaces. Is $X$ $E$-separable as well?
\end{prob}

\begin{prob}
How does $E$-separability behave under powers and products?
\end{prob}

\section{Acknowledgements}

A significant portion of the above presented research was done at the University of Toronto.
The authors would like to thank Franklin Tall and William Weiss for their helpful comments.
We thank the Set Theory and Topology group of the Alfr\'ed R\'enyi Institute of Mathematics (Hungarian Academy of Sciences) for further help in preparing this paper. We deeply thank Victoria Gitman for bringing to our attention Theorem \ref{joel}, which was imperative for our equiconsistency result.

 The first named author was partially supported by Capes (BEX-1088-11-4). The second named author was supported in part by the Ontario Trillium Scholarship and the FWF Grant I1921.

\end{document}